\documentclass[12pt]{amsart}
\usepackage{amssymb,bbold}
\usepackage[all]{xy}

\theoremstyle{plain}
\newtheorem{theorem}{Theorem}
\newtheorem{corollary}[theorem]{Corollary}
\newtheorem{lemma}[theorem]{Lemma}
\newtheorem{proposition}[theorem]{Proposition}

\theoremstyle{definition}

\newtheorem{remark}[theorem]{Remark}

\def\one{\mathbb 1}

\begin{document}
\baselineskip 18pt

\title[Unbounded continuous operators in Banach lattices]
      {Unbounded continuous operators in Banach lattices}

\author[O.~Zabeti]{Omid Zabeti}

\address[O.~Zabeti]
  {Department of Mathematics, Faculty of Mathematics, Statistics, and Computer science,
   University of Sistan and Baluchestan, Zahedan,
   P.O. Box 98135-674. Iran}
\email{o.zabeti@gmail.com}

\keywords{Unbounded continuous operator, $uaw$-continuous operator, adjoint of an operator, reflexive space, Banach lattice.}
\subjclass[2010]{Primary: 46B42. Secondary: 47B65.}

\begin{abstract}
Motivated by the equivalent definition of a continuous operator between Banach spaces in terms of weakly null nets, we introduce two types of continuous operators between Banach lattices using unbounded absolute weak convergence. We characterize reflexive Banach lattices in terms of these spaces of operators. Furthermore, we investigate whether or not the adjoint of these classes of operators has the corresponding property.
In addition, we show that these kinds of operators are norm closed but not order closed. Moreover, we establish how these classes of continuous operators are connected to the well-known classes of $M$-weakly compact operators and $L$-weakly compact operators.
Finally, we show that the notions of an $M$-weakly operator and a $uaw$-Dunford-Pettis operator have the same meaning; this extends one of the main results of Erkursun-Ozcan et al. (TJM, 2019).
\end{abstract}

\date{\today}

\maketitle
\section{motivation and introduction}
Let us first start with some motivation. Suppose $X$ and $Y$ are Banach spaces and $T:X\to Y$ is an operator. It is known that (see \cite[Theorem 5.22]{AB} for example), $T$ is continuous if and only if it preserves weakly null nets. On the other hand, unbounded convergences have received much attention recently with some deep and inspiring results ( see \cite{DOT, KMT, GTX, Z}). In particular, unbounded absolute weak convergence as a weak version of unbounded norm convergence and also as an unbounded version of weak convergence has been investigated in \cite{Z}. Furthermore, unbounded absolute weak Dunford-Pettis operators ( $uaw$-Dunford-Pettis operators, in brief), as an unbounded version of Dunford-Pettis operators, have been considered recently in \cite{EGZ}.
Before to proceed more, let us consider some preliminaries.

Suppose $E$ is a Banach lattice. A net $(x_{\alpha})$ in $E$ is said to be {\bf unbounded absolute weak convergent} ($uaw$-convergent, for short) to $x\in E$ if for each $u\in E_{+}$, $|x_{\alpha}-x|\wedge u\rightarrow0$ weakly, in notation $x_{\alpha}\xrightarrow{uaw}x$. $(x_{\alpha})$ is {\bf unbounded norm convergent} ($un$-convergent, in brief) if $\||x_{\alpha}-x|\wedge u\|\rightarrow 0$, in notation $x_{\alpha}\xrightarrow{un}x$. Both convergences are topological. For ample information on these concepts, see \cite{DOT, KMT, Z}.

Now, we consider the following observations as unbounded versions of continuous operators.

Suppose $E$ is a Banach lattice and $X$ is a Banach space. A continuous operator $T:E\to X$ is called {\bf unbounded continuous} if for each bounded net $(x_{\alpha})\subseteq E$, $x_{\alpha}\xrightarrow{uaw}0$ implies that $T(x_{\alpha})\xrightarrow{w}0$. Moreover, $T$ is said to be {\bf sequentially unbounded continuous} if for each bounded squence $(x_n)\subseteq E$, $x_n\xrightarrow{uaw}0$ implies that $T(x_n)\xrightarrow{w}0$.

Observe that a continuous operator $T:E\to F$, where $E$ and $F$ are Banach lattices, is said to be {\bf  $uaw$-continuous} if $T$ maps every norm bounded $uaw$-null net into a $uaw$-null net. It is {\bf  sequentially $uaw$-continuous} provided that the property happens for sequences. Consider this point that sequentially $uaw$-continuous operators were introduced in \cite{EGZ} at first as a beside note.

Moreover, recall that $T$ is $uaw$-Dunford-Pettis if for every norm bounded sequence $(x_n)\subseteq E$, $x_n\xrightarrow{uaw}0$ implies that $\|T(x_n)\|\rightarrow 0$; see \cite{EGZ} for a detailed exposition on this topic.

In this paper, our attempt is to investigate more about these classes of continuous operators. More precisely, we characterize reflexive Banach lattices in terms of these classes of continuous operators. Also, we consider several conditions under which the adjoint of an unbounded continuous ( a $uaw$-continuous) operator is again unbounded continuous ( $uaw$-continuous). Moreover, we investigate closedness properties for these spaces of operators. We shall show that these classes of operators are closed to the known classes of $M$-weakly compact operators and $L$-weakly compact ones. Among these, we improve one of the main results recently obtained in \cite{EGZ}; in particular, we show that $uaw$-Dunford-Pettis operators and $M$-weakly compact operators are in fact the same notions, without any extra conditions.

For undefined terminology and concepts, we refer the reader to \cite{AB}. All operators in this note, are assumed to be continuous, unless otherwise stated, explicitly.
\section{main result}
It is easy to see that every $uaw$-Dunford-Pettis operator is sequentially unbounded continuous but the converse is not true in general. The inclusion map from $c_0$ into $\ell_{\infty}$ is unbounded continuous but not $uaw$-Dunford-Pettis; suppose $(x_\alpha)$ is a bounded $uaw$-null net in $c_0$. By \cite[Theorem 7]{Z}, $x_{\alpha}\xrightarrow{w}0$ in $c_0$ so that in $\ell_{\infty}$. Now, observe that the standard basis $(e_n)$ is $uaw$-null but certainly not norm null.
\begin{theorem}\label{1}
Suppose $E$ is a Banach lattice. If every continuous operator $T:\ell_1\to E$ is unbounded continuous, then $E$ is reflexive.
\end{theorem}
\begin{proof}
Suppose every continuous operator is unbounded continuous. First we show that $E'$ is order continuous; suppose not. So, $E$ contains a lattice copy of $\ell_1$ ( with the embedding $\iota$). The identity operator $I$ on $\ell_1$ is continuous but not unbounded continuous. This implies that the composition map $T:\ell_1\to E$ defined via $T=\iota  oI$ is not unbounded continuous which contradicts our assumption.

Now, we prove that $E$ is a $KB$-space. Suppose on a contrary, it is not. So, it contains a lattice copy of $c_0$; with the embedding $\iota_1$. Consider the following diagram.
\[\ell_1\xrightarrow{T}L_1[0,1]\xrightarrow{S}c_0\xrightarrow{\iota_1}E,\]
in which $T(\alpha_1,\alpha_2,\ldots)=\Sigma_{n=1}^{\infty}\alpha_n {r_n}^{+}$ where $(r_n)$ is the sequence of Rademacher functions and $S$ is the classical "Fourier coefficients" due to Lozanovsky defined via $T(f)=(\int_{0}^{1}f(t)\sin t dt,\int_{0}^{1}f(t)\sin 2t dt,\ldots)$; for more details see \cite[Exercise 10, Page 289]{AB}. We claim that this composition is not unbounded continuous and this would complete our proof.
Suppose $(e_n)$ is the standard basis in $\ell_1$ which is certainly $uaw$-null. $Te_n={r_n}^+$ which is not weakly null. Moreover, $S({r_n}^+)$ is not also weakly null in $c_0$; since $\int_{0}^{1} {r_n}^{+}(t)\sin tdt \nrightarrow 0$. Now, it is obvious that $S({r_n}^+)$ is not weakly null in $E$.

\end{proof}
\begin{remark}\label{2}
The other implication of Theorem \ref{1} is not true in general. Consider operator $S:\ell_1 \to L_2[0,1]$ defined via $S(\alpha_1,\alpha_2,\ldots)=\Sigma_{n=1}^{\infty}\alpha_n {r_n}^{+}$. In which $(r_n)$ denotes the sequence of Rademacher functions; for more details about this operator consider \cite[Example 5.17, page 284]{AB}. Note that $S$ is not unbounded continuous; the standard basis $(e_n)$ is $uaw$-null in $\ell_1$ but $S(e_n)={r_n}^{+}$ which is not weakly null since $\int_{0}^{1}{r_n}^{+}(t)dt =\frac{1}{2}$.

But there is a good news if we apply $uaw$-continuity.
\end{remark}
\begin{proposition}
Suppose $E$ and $F$ are Banach lattices such that $F'$ is order continuous. Then every $uaw$-continuous operator $T:E\to F$ is unbounded continuous.
\end{proposition}
\begin{proof}
Suppose $T$ is $uaw$-continuous. For every norm bounded $uaw$-null net $(x_{\alpha})$ in $E$, $T(x_{\alpha})\xrightarrow{uaw}0$. By \cite[Theorem 7]{Z}, $T(x_{\alpha})\xrightarrow{w}0$, as claimed.
\end{proof}
\begin{corollary}\label{403}
Suppose $E$ is a Banach lattice and $F$ is an $AM$-space. Then it can be seen easily that an operator $T:E\to F$ is sequentially unbounded continuous if and only if it is sequentially $uaw$-continuous.
\end{corollary}
For the converse, we have the following.
\begin{proposition}
Suppose $E$ and $F$ are Banach lattices and $T:E\to F$ is a positive operator. If $T$ is unbounded continuous, then, it is $uaw$-continuous.
\end{proposition}
\begin{proof}
Suppose $(x_{\alpha})$ is a bounded positive $uaw$-null net in $E$. By the assumption, $T(x_{\alpha})\xrightarrow{w}0$. This implies that $T(x_{\alpha})\rightarrow0$ in the absolute weak topology. Therefore, $T(x_{\alpha})\xrightarrow{uaw}0$.
\end{proof}
\begin{corollary}
Suppose $E$ and $F$ are Banach lattices such that $F'$ possesses an order continuous norm. Then, a positive operator $T:E\to F$ is unbounded continuous if and only if it is $uaw$-continuous.
\end{corollary}

\begin{theorem}\label{6}
Suppose $E$ is a Banach lattice. If every continuous operator $T:L_1[0,1]\to E$ is $uaw$-continuous, then $E$ is reflexive.
\end{theorem}
\begin{proof}
  Suppose on a contrary, $E$ is not reflexive. First, assume that $E$ contains a lattice copy of $c_0$. Consider the operator $T:L_1[0,1]\to c_0$ defined via $T(f)=(\frac{\int_{0}^{1}f(t)dt}{n})$. We claim that $T$ is not $uaw$-continuous. Let $(f_n)$ be the bounded sequence in $L_1[0,1]$ defined as follows: $n$ on the interval $[0,\frac{1}{n}]$ and zero otherwise. Observe that by \cite[Theorem 4]{Z}, $uaw$-convergence and $un$-convergence in $L_1[0,1]$ agree and by using \cite[Corollary 4.2]{DOT}, we conclude that $f_n \xrightarrow{uaw}0$. But $T(f_n)=(x)$ in which $x$ is the constant sequence $(\frac{1}{n})$; not a $uaw$-null sequence. This implies that $T(f_n)$ is not also $uaw$-null in $E$ which is a contradiction.

  Now, suppose $E'$ is not order continuous so that $E$ contains a lattice copy of $\ell_1$. Observe that the operator $T:L_1[0,1]\to \ell_1$ defined via $T(f)=(\frac{\int_{0}^{1}f(t)dt}{n^2})$ is not $uaw$-continuous. Consider the sequence $(f_n)$ as the first part of argument. It is $uaw$-null. Nevertheless, $T(f_n)$ is the constant sequence $(y)$ defined via $y=(\frac{1}{n^2})$ which is certainly not $uaw$-null.
\end{proof}
Note that the other implication of Theorem \ref{6} may fail, as well. The operator $S:L_1[0,1]\to \ell_2$ defined by $S(f)=(\frac{\int_{0}^{1}f(t)dt}{n})$ is not $uaw$-continuous.

Now, we are going to investigate whether or not the converse of Theorem \ref{1} is true. First, we have the following lemma which can be proved via \cite[Theorem 7]{Z}.
\begin{lemma}\label{8}
Suppose $E$ is a Banach lattice. Then $E'$ is order continuous if and only if the identity operator $I$ on $E$ is unbounded continuous.
\end{lemma}
Furthermore, the following is immediate.
\begin{corollary}\label{100}
Suppose $E$ is a $KB$-space. Then the identity operator $I$ on $E$ is unbounded continuous if and only if $E$ is reflexive.
\end{corollary}

Observe that $KB$-space assumption is necessary in Corollary \ref{100} and can not be omitted. The identity operator $I$ on $c_0$ is unbounded continuous but $c_0$ is not reflexive.

Before we proceed with a kind of converse for Theorem \ref{1}, we provide some ideal properties.
\begin{proposition}\label{7}
 Let $S:E\to F$ and $T:F\to G$ be two operators between Banach lattices $E, F$, and $G$.
	\begin{itemize}
		\item[\em (i)] {If $T$ is Dunford-Pettis and $S$ is sequentially unbounded continuous then $TS$ is $uaw$-Dunford-Pettis}.
		\item[\em (ii)] {If $T$ is a $uaw$-Dunford--Pettis operator and $S$ is sequentially $uaw$-continuous, then $TS$ is $uaw$-Dunford-Pettis}.
		\item[\em (iii)]
	{ If $T$ is continuous and $S$ is an unbounded continuous operator, then $TS$ is also unbounded continuous}.
		\item[\em (iv)]{ If $T$ is an onto lattice homomorphism  and $S$ is $uaw$-continuous, then $TS$ is also $uaw$-continuous}.
	\end{itemize}
\end{proposition}
\begin{proof}
$(i)$. Suppose $(x_n)$ is a norm bounded $uaw$-null sequence in $E$. So, $S(x_n)\xrightarrow{w}0$. Therefore, $\|TS(x_n)\|\rightarrow 0$.

$(ii)$. Suppose $(x_n)$ is a norm bounded $uaw$-null sequence in $E$. Therefore, $S(x_n)\xrightarrow{uaw}0$. Thus, $\|TS(x_n)\|\rightarrow 0$.

$(iii)$. Suppose $(x_{\alpha})$ is a norm bounded $uaw$-null net in $E$. By assumption, $S(x_{\alpha})\xrightarrow{w}0$ so that $TS(x_{\alpha})\xrightarrow{w}0$.

$(iv)$. First, observe that for each $u\in G_{+}$, there is a $v\in F_{+}$ with $T(v)=u$. Suppose $(x_{\alpha})$ is a norm bounded $uaw$-null net in $E$.
By assumption, $S(x_{\alpha})\xrightarrow{uaw}0$ so that $|S(x_{\alpha})|\wedge v\xrightarrow{w}0$. Therefore,
\[|TS(x_{\alpha})|\wedge u= T(|S(x_{\alpha}|)\wedge u=T(|S(x_{\alpha})|\wedge v)\xrightarrow{w}0.\]
\end{proof}
Furthermore, we see that unbounded continuous operators are very closed to continuous operators. This is done using \cite[Theorem 7]{Z}.
\begin{lemma}\label{300}
Suppose $E$ is a Banach lattice whose dual space is order continuous and $F$ is any Banach lattice. Then every continuous operator $T:E\to F$ is unbounded continuous.
\end{lemma}
\begin{theorem}\label{402}
For a Banach lattice $E$, the following are equivalent.
\begin{itemize}
		\item[\em (i)] {$E'$ is order continuous}.
		\item[\em (ii)] {For any Banach lattice $F$, every continuous operator $T:E\to F$ is unbounded continuous}.
		\end{itemize}
\end{theorem}
\begin{proof}
$(i)\to(ii)$. This is done using Lemma \ref{300}.

$(ii)\to(i)$. By the assumption, The identity operator $I$ on $E$ is unbounded continuous and then use Lemma \ref{8}.
\end{proof}
Now, we have the following converse of Theorem \ref{1}.
\begin{corollary}\label{401}
For a $KB$-space $E$, the following are equivalent.
\begin{itemize}
		\item[\em (i)] {$E$ is reflexive}.
		\item[\em (ii)] {For any Banach lattice $F$, every continuous operator $T:E\to F$ is unbounded continuous}.
		\end{itemize}
\end{corollary}

\begin{remark}
Observe that in general, there are no relations between unbounded continuous operators and weakly compact ones. Consider \cite[Example 2.21]{EGZ}; the operator $T:\ell_1\rightarrow L_2[0,1]$ defined by $T(x_n)=(\sum_{n=1}^{\infty}x_n)\chi_{[0,1]}$ for all $(x_n)\in\ell_2$ where $\chi_{[0,1]}$ denotes the characteristic function of $[0,1]$. It is weakly compact but not unbounded continuous. Indeed, the standard basis $(e_n)$ in $\ell_1$ is $uaw$-null but $T(e_n)$ is not weakly null since $\int_{0}^{1}\chi_{[0,1]}dt=1$.
Moreover, the identity operator on $\ell_{\infty}$ is not weakly compact yet it is unbounded continuous using \cite[Theorem 7]{Z}.
\end{remark}
\begin{theorem}\label{3}
Suppose $E$ is a Banach lattice and $F$ is an order continuous Banach lattice. Then every weakly compact operator $T:E\to F$ has an unbounded continuous adjoint.
\end{theorem}
\begin{proof}
Assume that $({x_{\alpha}}')$ is a norm bounded net in $F'$ which is $uaw$-null. By \cite[Proposition 5]{Z},
${x_{\alpha}}'\xrightarrow{w*}0$. By the Gantmacher theorem \cite[Theorem 5.23]{AB}, $T'({x_{\alpha}}')\xrightarrow{w}0$, as desired.
\end{proof}
\begin{remark}\label{5}
Weakly compactness of operator $T$ and also order continuity of $F$  are essential in Theorem \ref{3} and can not be dropped. Consider the identity operator  $I:c_0\to c_0$. $I$ is not weakly compact but $F$ is order continuous. Furthermore, $I$ is also unbounded continuous. Its adjoint, $I:\ell_1\to \ell_1$ is not unbounded continuous; assume $(e_n)$ is the standard basis of $\ell_1$. It is $uaw$-null by \cite[Lemma 2]{Z}. But, certainly, it is not weakly null in $\ell_1$.

Also, consider the operator $T:L_2[0,1]\to \ell_{\infty}$ defined via $T(f)=(\int_{0}^{1}f(t)dt,\int_{0}^{1}f(t)dt,\ldots)$. It is weakly compact but $F$ is not order continuous. Consider the operator $T':(\ell_{\infty})'\to L_2[0,1]$. It is not unbounded continuous. Consider the standard basis $(e_n)$ which is $uaw$-null in $(\ell_{\infty})'$. But $<T'(e_n),\one>=<e_n,T(\one)>=1$.
\end{remark}

Suppose $E$ is a Banach lattice. The class of all unbounded continuous operators on $E$ is denoted by $B_{uc}(E)$, the class of all $uaw$-continuous operators on $E$ will get the terminology $B_{uaw}(E)$. In this step, we consider some closedness properties for these classes of continuous operators.
\begin{proposition}
Suppose $E$ is a Banach lattice. Then $B_{uc}(E)$ is closed as a subspace of the space of all continuous operators on $E$.
\end{proposition}
\begin{proof}
Suppose $(T_\alpha)$ is a net of unbounded continuous operators which is convergent to the operator $T$. We need to show that $T$ is unbounded continuous. For any $\varepsilon>0$, there is an $\alpha_0$ such that $\|T_{\alpha_0}-T\|<\frac{\varepsilon}{2}$. So, for each $x$ with $\|x\|\leq 1$, $\|T_{\alpha_0}(x)-T(x)\|<\frac{\varepsilon}{2}$. Assume that $(x_\beta)$ is a norm bounded $uaw$-null net in $E$. This means that $\|T_{\alpha_0}(x_{\beta})-T(x_{\beta})\|<\frac{\varepsilon}{2}$. Note that $T_{\alpha_0}(x_{\beta})\xrightarrow{w}0$ so that for a fixed $f\in E^{*}_{+}$ and sufficiently large $\beta$, $ f(T_{\alpha_0}(x_{\beta}))<\frac{\varepsilon}{2}$ so that $ f(T(x_{\beta}))<\varepsilon$.
\end{proof}
\begin{proposition}
Suppose $E$ is a Banach lattice. Then $B_{uaw}(E)$ is closed as a subspace of the space of all continuous operators on $E$.
\end{proposition}
\begin{proof}
Suppose $(T_\alpha)$ is a net of $uaw$-continuous operators which is convergent to the operator $T$. We need to show that $T$ is $uaw$-continuous. For any $\varepsilon>0$, there is an $\alpha_0$ such that $\|T_{\alpha_0}-T\|<\frac{\varepsilon}{2}$. So, for each $x$ with $\|x\|\leq 1$, $\|T_{\alpha_0}(x)-T(x)\|<\frac{\varepsilon}{2}$. Assume that $(x_\beta)$ is a norm bounded $uaw$-null net in $E$. This means that $\|T_{\alpha_0}(x_{\beta})-T(x_{\beta})\|<\frac{\varepsilon}{2}$. Note that $T_{\alpha_0}(x_{\beta})\xrightarrow{uaw}0$. Fix $f\in E^{*}_{+}$ and $u\in E_{+}$. Observe that for $a,b,c\geq 0$ in an Archimedean vector lattice, we have $|a\wedge c-b\wedge c|\leq |a-b|\wedge c$. Therefore,
\[f(|T_{\alpha_0}(x_{\beta})|\wedge u)-f(|T(x_\beta)|\wedge u)\leq \||T_{\alpha_0}((x_{\beta})|\wedge u)-(|T(x_\beta)|\wedge u)\|\]
\[\leq \|||T_{\alpha_0}(x_\beta)|-|T(x_{\beta})||\wedge u\|\leq \|||T_{\alpha_0}(x_\beta)|-|T(x_\beta)||\|\leq \|T_{\alpha}(x_{\beta})-T(x_{\beta})\|<\frac{\varepsilon}{2}.\]

 On the other hand, for sufficiently large $\beta$, $ f(T_{\alpha_0}(x_{\beta})\wedge u)\rightarrow 0$. This, in turn, results in $f(T(x_{\beta})\wedge u)\rightarrow 0$.
\end{proof}
\begin{remark}
Neither $B_{uc}(E)$ nor $B_{uaw}(E)$ are order closed, in general. Consider the operator $S$ in Remark \ref{2}. Put $S_n=SP_n$, in which $P_n$ is the canonical projection on $\ell_1$. Observe that each $S_n$ is finite rank. So, in the range space, $uaw$-convergence, weak convergence and norm one agree.
Therefore, we conclude that each $S_n$ is both $uaw$-continuous and unbounded continuous. Moreover $S_n\uparrow S$ and we have seen in Remark \ref{2} that $S$ is neither unbounded continuous nor $uaw$-continuous.

\end{remark}
\begin{remark}
Suppose $E$ and $F$ are Banach lattices and $T,S:E\to F$ are continuous operators such that $0\leq S\leq T$. It can be easily verified that if $T$ is either $uaw$-continuous or unbounded continuous, then so is $S$.
\end{remark}
\begin{theorem}
Suppose $E$ is a Banach lattice whose dual is order continuous and atomic and $F$ is an order continuous Banach lattice. Then every positive operator $T:E\to F$ has a $uaw$-continuous adjoint.
\end{theorem}
\begin{proof}
  Consider positive operator $T':F'\to E'$. Suppose $({x_{\alpha}}')$ is a positive bounded net in $F'$ such that ${x_{\alpha}}'\xrightarrow{uaw}0$. By \cite[Proposition 5]{Z}, ${x_{\alpha}}'\xrightarrow{w^*}0$ so that $T'{x_{\alpha}}'\xrightarrow{w^*}0$. By \cite[Proposition 8.5]{KMT} and \cite[Theorem 7]{Z}, we see that $T'{x_{\alpha}}'\xrightarrow{uaw}0$, as claimed.
\end{proof}
Furthermore, when the operator $T$ is not positive, we may consider the following.
\begin{proposition}
Suppose $E$ is an order continuous Banach lattice whose dual is also order continuous and atomic and $F$ is an order continuous Banach lattice. Then every continuous operator $T:E\to F$ has a $uaw$-continuous adjoint.
\end{proposition}
\begin{proof}
Consider operator $T':F'\to E'$. Suppose $({x_{\alpha}}')$ is a positive bounded net in $F'$ such that ${x_{\alpha}}'\xrightarrow{uaw}0$. By \cite[Proposition 5]{Z}, ${x_{\alpha}}'\xrightarrow{w^*}0$ so that $T'{x_{\alpha}}'\xrightarrow{w^*}0$. By \cite[Theorem 8.4]{KMT} and \cite[Theorem 7]{Z}, we see that $T'{x_{\alpha}}'\xrightarrow{uaw}0$, as claimed.
\end{proof}
\begin{theorem}\label{2020}
Suppose $E$ is a Banach lattice. If every operator $T:E\to \ell_{\infty}$ is $uaw$-Dunford-Pettis, then $E$ is reflexive.
\end{theorem}
\begin{proof}
Suppose not. So, $E$ contains a lattice copy of $c_0$ or $\ell_1$. For the former case, $E$ contains a disjoint positive sequence $(x_n)$ which is not norm convergent. By \cite[Lemma 3.4]{AEH}, there exists a bounded positive disjoint sequence $(g_n)$ in $E'$  such that $g_n(x_n)=1$ and $g_n(x_m)=0$ for all $m \neq n$. Define the operator $T:E\to \ell_{\infty}$ via $T(x)=(g_r(x))_{r\in \Bbb N}$. Note that by \cite[Lemma 2]{Z}, $x_n\xrightarrow{uaw}0$ but $\|T(x_n)\|\geq |g_n(x_n)|=1$ so that $\|T(x_n)\|\nrightarrow 0$. This shows that $T$ is not $uaw$-Dunford-Pettis.

For the latter case, $E$ contains a positive bounded disjoint sequence $(x_n)$ which is not weakly null. By a similar argument as we had at the first part of the proof, we get another contradiction, too.

\end{proof}
Note that the converse of Theorem \ref{2020} is not true, in general. The inclusion map $\iota:\ell_2\to \ell_{\infty}$ is not $uaw$-Dunford-Pettis. However, when we consider sequentially unbounded continuous operators, we have better news. More precisely, we present another statement similar to Theorem \ref{402}.
\begin{theorem}\label{2021}
Suppose $E$ is a Banach lattice. For every operator $T:E\to \ell_{\infty}$, the following assertions are equivalent.
\begin{itemize}
		\item[\em (i)] {$E'$ is order continuous}.
	\item[\em (ii)] {Every continuous operator $T:E\to \ell_{\infty}$ is sequentially unbounded continuous}.
		\end{itemize}
\end{theorem}
\begin{proof}
$(i)\to (ii)$. It is a direct consequence of Theorem \ref{402}.

$(ii)\to (i)$. Suppose not. So, $E$ contains a lattice copy of $\ell_1$. Therefore, $E$ possesses a positive bounded disjoint sequence $(x_n)$ which is not weakly null.
By \cite[Lemma 3.4]{AEH}, there exists a bounded positive disjoint sequence $(g_n)$ in $E'$  such that $g_n(x_n)=1$ and $g_n(x_m)=0$ for all $m \neq n$. Define the operator $T:E\to \ell_{\infty}$ by $T(x)=(g_1(x),(g_1+g_2)(x),\ldots)$. $T$ is well-defined because the sequence $(g_i)$ is disjoint and bounded. Observe that $x_n\xrightarrow{uaw}0$ by \cite[Lemma 2]{Z} but $T(x_n)=(0,\ldots,0,1,1,\ldots)$ in which the zero is appeared $n$-times. By using Dini's theorem (\cite[Theorem 3.52]{AB}), $T(x_n)\nrightarrow 0$ weakly and so by \cite[Theorem 7]{Z}, $T(x_n)\nrightarrow 0$ in the $uaw$-topology, a contradiction.

\end{proof}
Moreover, by considering Corollary \ref{403}, we obtain the following.
\begin{corollary}
Suppose $E$ is a Banach lattice. For every operator $T:E\to \ell_{\infty}$, the following assertions are equivalent.
\begin{itemize}
		\item[\em (i)] {$E'$ is order continuous}.
	\item[\em (ii)] {Every continuous operator $T:E\to \ell_{\infty}$ is sequentially $uaw$-continuous}.
		\end{itemize}
\end{corollary}
In the following, we prove one of the main results of \cite{EGZ} under less hypotheses; more precisely, we show that $uaw$-Dunford-Pettis operators and $M$-weakly compact ones are in fact the same.

Recall that an operator $T:E\to X$, where $E$ is a Banach lattice and $X$ is a Banach space, is called $M$-weakly compact if for every norm bounded disjoint sequence $(x_n)$ in $E$, $\|T(x_n)\|\rightarrow 0$. Moreover, $T$ is said to be $o$-weakly compact if $T[0,x]$ is a weakly relatively compact set for every $x\in E_{+}$. Also, note that a continuous operator $T:X\to E$ is said to be $L$-weakly compact if every disjoint sequence in the solid hull of $T(B_X)$ is norm null.
\begin{theorem}\label{400}
Suppose $E$ is a Banach lattice and $X$ is a Banach space. If $T:E\to X$ is an $M$-weakly compact operator then it is $uaw$-Dunford-Pettis.
\end{theorem}
\begin{proof}
Suppose $(x_n)$ is a positive norm bounded $uaw$-null sequence in $E$. This means that $x_n\wedge u\xrightarrow{w}0$ for any positive $u\in E$. Observe that $T$ is also $o$-weakly compact ( by \cite[Theorem 5.57]{AB} due to Dodds). So, using \cite[Exercise 3, Page 336]{AB}, convinces us that $\|T(x_n\wedge u)\|\rightarrow 0$. Note that by \cite[Theorem 5.60]{AB} due to Meyer-Nieberg, the proof would be complete.
\end{proof}
By using Theorem \ref{400} and \cite[Proposition 2.6]{EGZ}, we conclude that the notions of an $M$-weakly compact operator and a $uaw$-Dunford-Pettis operator agree. Therefore, many results regarding $uaw$-Dunford-Pettis operators mentioned in \cite{EGZ} can be restated directly for $M$-weakly compact operators.
Finally, we shall show that unbounded continuous operators are very closed to $M$-weakly compact operators and $uaw$-continuous operators are related to $L$-weakly compact operators. Before we proceed, we consider the following definition.

Suppose $E$ is a Banach lattice and $X$ is a Banach space. A continuous operator $T:E\to X$ is said to be  $WM$-weakly compact if for every bounded disjoint sequence $(x_n)\subseteq E$, $T(x_n)\xrightarrow{w}0$. A continuous operator $T:X\to E$ is called a $WL$-weakly compact operator if $y_n\xrightarrow{w}0$ for each disjoint sequence $(y_n)$ in the solid hull of $T(B_X)$. It is easy to see that every $M$-weakly compact operator is $WM$-weakly compact and every $L$-weakly compact operator is a $WL$-weakly compact operator. Nevertheless, the inclusion map from $c_0$ into $\ell_{\infty}$ is both $WM$-weakly compact  and $WL$-weakly compact, although, it is neither an $M$-weakly compact operator nor an $L$-weakly compact operator.

These classes of continuous operators enjoy an approximation property similar to \cite[Theroem 5.60]{AB}.
\begin{lemma}\label{500}
For a Banach lattice $E$ and a Banach space $X$, the following assertions hold.
\begin{itemize}
\item[\em (i)] {If $T:E\to X$ is a $WM$-weakly compact operator, then for each $\varepsilon>0$ and for each $f\in X'$, there exists some $u\in E_{+}$ such that $f(T(|x|-u)^{+})<\varepsilon$ satisfies for all $x\in E$ with $\|x\|\leq 1$}.
		\item[\em (ii)] {If $T:X\to E$ is a $WL$-weakly compact operator, then for each $\varepsilon>0$ and for each $f\in E'$, there exists some $u\in E_{+}$ such that $f((|Tx|-u)^{+})<\varepsilon$ satisfies for all $x\in X$ with $\|x\|\leq 1$}.
\end{itemize}
\end{lemma}
\begin{proof}
$(i)$. Suppose $B$ is the closed unit ball of $E$ and $(x_n)$ is a disjoint sequence in $B$. Given $\varepsilon>0$ and $f\in X'$. By the assumption, $f(Tx_n)\rightarrow 0$. By \cite[Theorem 4.36]{AB} ( with $\rho(x)=f(|x|)$), there exists $u\in E_{+}$ with $f(T(|x|-u)^{+})<\varepsilon$ for all $x$ with $\|x\|\leq 1$.

$(ii)$. Suppose $U$ is the closed unit ball of $X$ and $A$ is the solid hull of $T(U)$. $T$ is $WL$-weakly compact so that every disjoint sequence in $A$ is weakly null. Consider the identity operator $I$ on $E$ and fix $f\in E'$. Again, by considering $\rho(x)=f(|x|)$ and using \cite[Theorem 4.36]{AB}, we see that $f(I(|y|-x)^{+})<\varepsilon$ for each $y\in A$ so that $f((|Tx|-u)^{+})<\varepsilon$ for all $x\in U$.
\end{proof}
Therefore, we have the following. Just, note that in the definition of an unbounded continuous operator, we can replace the second space with a Banach space.
\begin{theorem}
Suppose $E$ is a Banach lattice and $X$ is a Banach space. Then a continuous operator $T:E\to X$ is sequentially unbounded continuous if and only if it is a $WM$-weakly compact operator.
\end{theorem}
\begin{proof}
The direct implication is done using \cite[Lemma 2]{Z}. For the other side, assume that $T:E\to X$ is $WM$-weakly compact and $(x_n)$ is a bounded $uaw$-null sequence in $E$. Given $\varepsilon>0$ and $f\in X'$. Considering Lemma \ref{500}, we can find $u\in E_{+}$ such that $f(T(|x_n|-u)^{+})<\varepsilon$ for all $n\in \Bbb N$. Note that $x_n\wedge u\xrightarrow{w}0$ so that $T(x_n\wedge u)\xrightarrow{w}0$. Therefore, $T(x_n)\xrightarrow{w}0$, as claimed.
\end{proof}
Now, we state a version of \cite[Theorem 2.10]{EGZ} for the case of unbounded continuous operators.
\begin{proposition}
Suppose $E$ and $F$ are Banach lattices. Then every $WL$-weakly compact sequentially $uaw$-continuous operator $T:E\to F$ is sequentially unbounded continuous.
\end{proposition}
\begin{proof}
Suppose $(x_n)$ is a bounded $uaw$-null sequence in $E$. Given $\varepsilon>0$ and $f\in F'$. By Lemma \ref{500}, there exists $u\in F_{+}$ such that $f((|T(x_n)|-u)^{+})<\varepsilon$. Note that $|T(x_n)|\wedge u\xrightarrow{w}0$. Thus, $T(x_n)\xrightarrow{w}0$, as desired.
\end{proof}
By using \cite[Theorem 4.34]{AB}, we obtain the following surprising but simple result.
\begin{proposition}
Suppose $X$ is a Banach space and $E$ is a Banach lattice. Then every weakly compact operator $T:X\to E$ is $WL$-weakly compact.
\end{proposition}
For the converse, we have the following.
\begin{theorem}\label{200}
Suppose $E$ is a Banach lattice with a strong unit and $F$ is an order continuous Banach lattice. Then every positive $WL$-weakly compact operator $T:E\to F$ is  weakly compact.
\end{theorem}
\begin{proof}
Since $E$ is a $C(K)$-space for some compact Hausdorff space $K$, order boundedness and boundedness notions agree in $E$. So, if $U$ denotes the closed unit ball of $E$, then it is order bounded. Therefore $T(U)$, and so, $Sol(T(U))$ ( the solid hull of $T(U)$), is also order bounded. Now, suppose $(y_n)$ is a disjoint bounded sequence in $Sol(T(U))$. Therefore, it is order bounded and disjoint so that weakly null. By \cite[Theorem 4.17]{AB}, it is norm null. Now, by \cite[Theorem 5.55]{AB}, $Sol(T(U))$ is weakly relatively compact. Therefore, $T(U)$ is also weakly relatively compact.
\end{proof}
Note that hypotheses in Theorem \ref{200} are essential and can not be removed. Consider the identity operator on $\ell_{\infty}$. Using \cite[Lemma 2 and Theorem 7]{Z}, we conclude that it is $WL$-weakly compact. But it is not weakly compact, definitely. Moreover, consider the identity operator on $c_0$. With the same argument, we see that it is  $WL$-weakly compact but certainly not a weakly compact operator.

\end{document}